\documentclass[11pt]{article}

\newcommand\version{May 27, 2026}

\usepackage{amsmath,amssymb,amsthm,mathtools}

\newtheorem{theorem}{Theorem}
\newtheorem{lemma}{Lemma}

\newcommand{\avg}[1]{\left\langle #1\right\rangle}
\newcommand{\Ent}{\operatorname{Ent}}
\newcommand{\E}{\mathcal E}
\newcommand{\ip}[1]{\left\langle #1\right\rangle}

\newcommand{\one}{\mathbf 1}
\newcommand{\Var}{\operatorname{Var}}
\newcommand{\Z}{\mathbb Z}

\title{The sharp log Sobolev inequality on finite cycles}
\author{Rupert L. Frank\footnote{LMU Munich, r.frank@lmu.de} \footnote{Munich Center for Quantum Science and Technology} \footnote{Caltech}
\and Paata Ivanisvili\footnote{University of California, Irvine, pivanisv@uci.edu}}
\date{\version}

\begin{document}
	\maketitle
	
	\begin{abstract}
		We settle the problem of finding the sharp constant in the log Sobolev inequality on the $n$-cycle for all $n\ge 4$, by showing that it is equal to half of the spectral gap. We deduce this result from an optimal cubic Sobolev inequality.
	\end{abstract}

%%%%%%%%%%%%%%%%%%%%%%%

\section{Introduction and main result}

Let
$$
C_n := \Z/n\Z \,,\quad n\ge 2 \,,
$$
the $n$-cycle. We write functions $f:C_n\to\mathbb R$ as $f=(f_i)_{i\in C_n}$ and use the normalized average
\[
\avg{f} :=\frac1n\sum_{i\in C_n} f_i.
\]
The natural Dirichlet form is
\begin{equation}\label{eq:dirichlet}
	\E_n(f,f):=\frac1{2n}\sum_{i\in C_n}(f_i-f_{i+1})^2
	=\frac12\avg{(f_i-f_{i+1})^2}.
\end{equation}
The variance is
$$
\Var(f)=\avg{f^2}-\avg{f}^2
$$
and, for a nonnegative function $g:C_n\to\mathbb R$, the relative entropy is
\[
\Ent(g):=\avg{g\log g}-\avg{g}\log\avg{g}
\]
with the convention $0\log0=0$. The two constants we are interested in are the spectral gap
\begin{equation}\label{eq:lambda_var}
	\lambda_n
	:=\inf_{\Var(f)>0}\frac{\E_n(f,f)}{\Var(f)}
\end{equation}
and the logarithmic Sobolev constant
\begin{equation}\label{eq:alpha}
	\alpha_n
	:=\inf_{f:\,\Ent(f^2)>0}
	\frac{\E_n(f,f)}{\Ent(f^2)} \,.
\end{equation}
In this, as in many other problems, the computation of $\alpha_n$ is much harder than that of $\lambda_n$. Indeed, a standard Fourier computation (as, for instance, in Lemma \ref{lem:highfreq} below) gives
\begin{equation}\label{eq:lambda}
	\lambda_n=1-\cos\frac{2\pi}{n} \,.
\end{equation}
Also, a well-known and rather general argument gives the bound
\begin{equation}
	\label{eq:upperbound}
	\alpha_n\le \frac{\lambda_n}{2}.
\end{equation}
Our main result says that in this inequality one has equality.

\begin{theorem}
	\label{thm:main}
	Let $n\ge 4$. Then
	$$
	\boxed{\alpha_n = \frac{\lambda_n}{2}.}
	$$
\end{theorem}

\noindent
\emph{Remarks.}
(a) Some cases of this theorem were known before: Chen and Sheu \cite{ChenSheu2003} proved the theorem for all \emph{even} $n\geq 4$, Chen, Liu and Saloff-Coste \cite{ChenLiuSaloffCoste2008} proved it for $n=5$ and, more recently, Faust and Fawzi \cite{FaustFawzi2021} proved it for odd $n\in\{7,\ldots,21\}$ using certified numerics. Our proof recovers all these results in a unified way, using a completely different approach.\\
	(b) Interest in the question of determining the sharp value of $\alpha_n$ was sparked by the influential work of Diaconis and Saloff-Coste \cite{DiaconisSaloffCoste1996}. Among other things, they showed that $\alpha_n$ is of the same order as $\lambda_n$. They also computed $\alpha_3$ and concluded that $\alpha_3 < \frac{\lambda_3}{2}$. In this sense our assumption $n\ge 4$ is best possible. For $n=2$ one has $\alpha_2=\frac{\lambda_2}{2}$, as is well-known.\\
	(c) The log Sobolev and the spectral gap inequality arise naturally in connection with the simple random walk on $C_n$. The corresponding Markov kernel $K$ is given by $K(i,j)=1/2$ if $j=i+1$ or $j=i-1$ and $K(i,j)=0$ otherwise, and the uniform distribution on $C_n$ is its unique stationary distribution.\\
	(d) In the limit $n\to\infty$ one obtains the sharp log Sobolev inequality on the (continuous) circle \cite{Rothaus1980},
	$$
	\int_0^1 (f')^2\,d\theta \geq \frac{(2\pi)^2}{2} \left( \int_0^1 f^2 \log f^2\,d\theta - \int_0^1 f^2\,d\theta \, \log \int_0^1 f^2\,d\theta \right),
	$$
    valid for all functions $f\in H^1(\mathbb R /\mathbb Z)$. Indeed, it suffices to apply the inequality on $C_n$ to the restriction of $f$ to $\{i/n:\ i=0,\ldots,n-1\}$. Alternatively, one can note that our proof extends to the continuum and thus gives another proof of this inequality.\\
	(e) The topic of log Sobolev inequalities has a long history, which we do not attempt to survey, except for mentioning the ground breaking paper \cite{Gross1975} and the review \cite{Gross2006}. For the discrete setting, we refer to the already mentioned works \cite{DiaconisSaloffCoste1996,ChenLiuSaloffCoste2008}.\\
    (f) Using the usual tensorization argument for log Sobolev inequalities \cite{Gross2006}, we obtain the optimal log Sobolev constant for
    $$
    C_{n_1}\times \cdots \times C_{n_L}
    $$
    with Dirichlet form
    $$
    \mathcal E(f,f) = \sum_{\ell=1}^L c_\ell \, \mathcal E^{(n_\ell)}(f,f) \,,
    $$
    where $n_\ell\neq 3$ for all $\ell\in\{1,\ldots,L\}$, where $c_\ell>0$ are arbitrary coefficients and where $\mathcal E^{(n_\ell)}$ only acts on the $\ell$-th factor. The sharp constant is
    $$
    \min_{1\leq\ell\leq L} c_\ell \, \frac{\lambda_{n_\ell}}{2} \,.
    $$
    \\
    (g) Again by standard facts about log Sobolev inequalities \cite{Gross2006} we obtain the hypercontractivity bound
    $$
    \| P_t f \|_q \leq \|f\|_p
    \qquad\text{for}\ e^{-2\lambda_n t} \leq \frac{p-1}{q-1}
    $$
    for $n\geq 4$, where $P_t$ is the simple random walk semigroup on $C_n$. A corresponding result is valid for products of cycles as in (f).

%%%%%%%%%%%%%%%%%%%%%%%

\subsection*{The cubic Sobolev inequality and the reduction of Theorem \ref{thm:main}}

The key ingredient in our proof is the following sharp cubic Sobolev inequality, where we set
$$
D(x):=\ip{(x_j-x_{j+1})^2} = 2\, \E_n(x,x) \,.
$$

\begin{theorem}[Cubic Sobolev inequality]
	\label{thm:cubic}
	Let \(n\ge4\). If \(x_j\ge0\) for all \(j\in C_n\) and
	\[
	\ip{x^2}=1,
	\]
	then
	\[
	\boxed{
        D(x) \geq \frac{2\lambda_n}{3}
		\ip{(x-1)^2(x+2)} \,.
	}
	\]
\end{theorem}

\noindent
\emph{Remarks.} 
(a) Note that if $x=\one$ (that is, $x_j=1$ for all $j$), then both sides of the cubic Sobolev inequality vanish. Our proof will show that there is a two-dimensional subspace $V_1$, orthogonal to constants, such that for $v\in V_1$ and $x=(\one+\epsilon v)/\sqrt{1+\epsilon^2 \ip{v^2}}$, the inequality is saturated to order $\epsilon^2$ as $\epsilon\to 0$. This is a common feature of our new cubic Sobolev inequality and the original log Sobolev inequality.\\
(b) Similarly to Theorem \ref{thm:main} one can pass to the limit $n\to\infty$ and obtains the inequality
$$
\int_0^1 (f')^2\,d\theta \geq \frac{(2\pi)^2}{3} \int_0^1 (f-1)^2(f+2) \,d\theta \,,
$$
valid for all $0\leq f\in H^1(\mathbb R/\mathbb Z)$ with $\int_0^1 f^2\,d\theta =1$. We have not seen this inequality before in the literature.\\
(c) As we will explain below, our proof of Theorem \ref{thm:cubic} shows some similarities with the proof of degenerate stability of some Sobolev-type inequalities in \cite{Fra2022}. Probably, with some additional work one could get a quantitative version of Theorem \ref{thm:cubic} similarly as in \cite{Fra2022,BriDolSim2023}; see also \cite{DolEstFigFraLos2025} for an optimal stability result for a log Sobolev inequality.

\medskip

The main part of this paper will be concerned with the proof of Theorem~\ref{thm:cubic}. Before diving into this, however, we give the quick argument that allows to deduce Theorem \ref{thm:main} from Theorem \ref{thm:cubic}.

We show that the logarithmic integrand is bounded above by the cubic polynomial
$$
P_3(t):=2(t-1)+3(t-1)^2+\frac23(t-1)^3.
$$

\begin{lemma}[Cubic majorant]\label{lem:cubic-majorant}
	For every $t>0$,
	\begin{equation}\label{eq:cubic-majorant}
		2t^2\log t \le P_3(t) \,.		
	\end{equation}
\end{lemma}

\begin{proof}
	For $t>0$, set
	\[
	H(t) :=P_3(t)-2t^2\log t.
	\]
	Then $H(1)=H'(1)=H''(1)=H'''(1)=0$, and
	\[
	H^{(4)}(t)=\frac4{t^2}>0.
	\]
	Thus, for every $t>0$
	$$
	H(t)=\frac16\int_1^t H^{(4)}(s)(t-s)^3\,ds
	=\frac23\int_1^t\frac{(t-s)^3}{s^2}\,ds\ge0.
	$$
	Here, for the final inequality, we distinguish the cases $t>1$ and $t<1$.
\end{proof}

The cubic polynomial has the useful identity
\begin{equation}\label{eq:P3-identity}
	P_3(t)=\frac23(t-1)^2(t+2)+(t^2-1).
\end{equation}

\begin{proof}[Proof of Theorem \ref{thm:main}]
	In the infimum \eqref{eq:alpha} defining $\alpha_n$ it suffices to consider $f\ge0$, because replacing $f$ by $|f|$ leaves $\Ent(f^2)$ unchanged and does not increase $\E_n(f,f)$.  By homogeneity, normalize $\avg{f^2}=1$ and write $x=f$.  Lemma \ref{lem:cubic-majorant} gives
	\[
	\Ent(x^2)=\avg{2x^2\log x}
	\le \avg{P_3(x)}.
	\]
	Using \eqref{eq:P3-identity} and $\avg{x^2}=1$,
	\[
	\avg{P_3(x)}=\frac23\avg{(x-1)^2(x+2)}.
	\]
	Combining these two equations with Theorem \ref{thm:cubic}, we obtain
	\[
	\Ent(x^2)
	\le
	\frac{1}{\lambda_n}\, D(x)	=
	\frac{2}{\lambda_n}\E_n(x,x).
	\]
	Hence $\alpha_n\ge \lambda_n/2$. The reverse bound comes from \eqref{eq:upperbound}.	
\end{proof}

%%%%%%%%%%%%%%%%%%%%%%%%%%%%%%%%%%%%%%%%%%

\subsection*{Outline of the proof of Theorem \ref{thm:cubic}}

As we have mentioned after Theorem \ref{thm:cubic}, our cubic Sobolev inequality is saturated for functions of the form $x=(\one+\epsilon v)/\sqrt{1+\epsilon^2 \ip{v^2}}$ where $v$ belongs to a certain two-dimensional subspace $V_1$. It is therefore natural to decompose an arbitrary function $x$ as
$$
x = \ip{x} + v + w \,,
$$
where $v\in V_1$ and where $w$ is orthogonal both to constants and to $V_1$. The important observation is that the spectral gap for $w$ is strictly greater than the spectral gap for arbitrary functions, and it is this coercivity that allows us to prove the bound.

Of course, in the nonlinearity $(x-1)^2(x+2)$ there will be cross terms between the contributions of $\ip{x}$, $v$ and $w$ and a lot of our work goes into controlling these `interactions'. It is at this point that the algebraic nonlinearity $(x-1)^2(x+2)$ is much more convenient than the original nonlinearity $x^2\log x$. 

We expect that the idea of replacing $x^2\log x$ by the simpler nonlinearity $(x-1)^2(x+2)$ might be useful in other questions related to log Sobolev inequalities as well. Note that this idea also appears in \cite{FaustFawzi2021} for a different purpose, which forces them to use a more complicated fifth order polynomial.

%%%%%%%%%%%%%%%%%%%%%%%%%%%%%%%%%
%%%%%%%%%%%%%%%%%%%%%%%%%%%%%%%%%

	\section{Proof of the cubic Sobolev inequality}

	\subsection*{Scalar inequalities}
	
	\begin{lemma}
		\label{lem:scalar1}
		Let \(a,r,t\ge0\) satisfy
		\[
		a^2+r^2+t^2=1.
		\]
		Then
		\begin{equation}
			\label{eq:scalar1}
					\frac{3}{\sqrt2}r^2t+3\sqrt2\,rt^2
			\le
			(1-a)^2(1+2a)+4t^2 \,,
		\end{equation}
		\begin{equation}
			\label{eq:scalar2}
		\frac{3}{\sqrt2}(r^2t+rt^2)
		\le
		(1-a)^2(1+2a)+\frac52t^2
		\end{equation}
        and
        \begin{equation}
			\label{eq:scalar3}
		3 r^2t
		\le
		(1-a)^2(1+2a)+3t^2 \,.
		\end{equation}
	\end{lemma}
	
	\begin{proof}
        We begin by noting that
        $$
        (1-a)^2(1+2a) - \frac34(1-a^2)^2 = \frac14 (1-a)^3(1+3a) \ge 0 \,,
        $$
        so
        $$
        (1-a)^2(1+2a) \ge \frac34(1-a^2)^{2} = \frac34(r^2+t^2)^{2} \,.
        $$
		Thus it is enough to prove
		\[
		\frac{3}{\sqrt2}r^2t+3\sqrt2\,rt^2
		\le
		\frac34(r^2+t^2)^2+4t^2 \,,
		%\tag{3}
		\]
		\[
		\frac{3}{\sqrt2}(r^2t+rt^2)
		\le
		\frac34(r^2+t^2)^2+\frac52t^2
		%\tag{5}
		\]
        and
        \[
        3r^2t \le \frac34(r^2+t^2)^2+3t^2 \,.
        \]        
		To do so, we may assume \(t>0\) and put
		\[
		s=\frac rt \,.
		\]
		After dividing by \(t^2\), the claimed inequalities become
		\begin{equation}
			\label{eq:goal1}
		\frac34(s^2+1)^2t^2
		-\frac3{\sqrt2}s(s+2)t
		+4 \geq 0 \,,
		\end{equation}
		\begin{equation}
			\label{eq:goal2}
					\frac34(s^2+1)^2t^2
			-
			\frac3{\sqrt2}s(s+1)t
			+
			\frac52
			\ge0
		\end{equation}
        and
        \begin{equation}
            \label{eq:goal3}
            \frac34(s^2+1)^2t^2 - 3s^2 t+3 \ge 0 \,.
        \end{equation}
        
		All three inequalities are quadratic in \(t\), so it suffices to show that the discriminant $\Delta$ is negative. We now treat the three cases separately.
		
		The discriminant in the first case is
		\[
		\Delta
		=
		\frac92s^2(s+2)^2
		-
		12(s^2+1)^2.
		\]
		Let \(\varphi=(1+\sqrt5)/2\). Since
		\[
		\varphi(s^2+1)-s(s+2)
		=
		\frac1\varphi(s-\varphi)^2\ge0,
		\]
		we have
		\[
		s(s+2)\le \varphi(s^2+1).
		\]
		Because \(3\varphi^2<8\),
		\[
		3s^2(s+2)^2
		\le
		3\varphi^2(s^2+1)^2
		<
		8(s^2+1)^2.
		\]
		Thus \(\Delta<0\), so the quadratic is nonnegative for all \(t\), proving \eqref{eq:goal1}.
		
		The discriminant in the second case is
		\[
		\Delta
		=
		\frac92s^2(s+1)^2
		-
		\frac{15}{2}(s^2+1)^2.
		\]
		Let \(\sigma=1+\sqrt2\). Since
		\[
		\frac{\sigma}{2}(s^2+1)-s(s+1)
		=
		\frac1{2\sigma}(s-\sigma)^2\ge0,
		\]
		we have
		\[
		s(s+1)\le \frac{\sigma}{2}(s^2+1).
		\]
		Since
		\[
		3\left(\frac{\sigma}{2}\right)^2<5,
		\]
		we get
		\[
		3s^2(s+1)^2
		\le
		3\left(\frac{\sigma}{2}\right)^2(s^2+1)^2
		<
		5(s^2+1)^2.
		\]
		Thus \(\Delta<0\), so the quadratic is nonnegative for all \(t\), proving \eqref{eq:goal2}.

        Finally, the discriminant in the third case is
        $$
        \Delta = 9s^4 - 9 (s^2+1)^2 \,.
        $$
        Clearly, \(\Delta<0\), so the quadratic is nonnegative for all \(t\), proving \eqref{eq:goal3}.
	\end{proof}
	
%%%%%%%%%%%%%%%%%%%%%%%%%%%%%%%%%%%%%%
	
	\subsection*{Fourier preliminaries}
	
	The graph Laplacian $L$ is defined by
	\[
	(Lx)_j=2x_j-x_{j-1}-x_{j+1} \,.
	\]
	This is relevant in our context since
	$$
	D(x)= \avg{x(Lx)} \,.
	$$
	The graph Laplacian has eigenvalues
	\[
	\mu_k:=2\left(1-\cos\frac{2\pi k}{n}\right) \,,\quad k=0,\ldots,n-1 \,.
	\]
	with eigenfunctions $\chi^{(k)}:C_n\to\mathbb C$ given by
	\[
	\chi^{(k)}_j:=e^{2\pi i k j/n} \,, \quad k=0,\ldots,n-1 \,.
	\]
    In particular, there is a simple zero eigenvalue $\mu_0=0$ with eigenfunction $\chi^{(0)}=\one$ and the first nonzero eigenvalue is
	\[
	\mu_1=\mu_{n-1}=2\lambda_n
	\]
	and the corresponding eigenspace is
	\[
	V_1:=\operatorname{span}
	\left\{
	\cos\frac{2\pi j}{n},
	\sin\frac{2\pi j}{n}
	\right\}.
	\]
    It follows that the quadratic form
    $$
    Q(x) := \frac1{\lambda_n} D(x) - 2 \|x\|_2^2
    $$
    is nonpositive when $x$ is a constant, is zero when $x$ is in $V_1$ and is nonnegative when $x$ is orthogonal to constants and to $V_1$. The next lemma quantifies this nonnegativity.

    We use the standard notations
    $$
    \| x \|_2 = |\avg{|x|^2}|^{1/2} \,,
    \qquad
    \| x \|_\infty = \max_{i\in C_n} |x_i| \,.
    $$

	\begin{lemma}[High-frequency estimate]
		\label{lem:highfreq}
		Let \(n\ge 4\), and let \(z\perp \one,\, V_1\).
		Then
		\begin{equation}
			\label{eq:linftybound}
			Q(z) \geq \sigma_n^{-1} \, \|z\|_\infty^2 
			\qquad
			\sigma_n
			:=
			\frac34-\frac14\tan^2\frac\pi n \,,
		\end{equation}
		and
		\begin{equation}
			\label{eq:gapbound}
		Q(z)\ge \kappa_n \, \|z\|_2^2 \,,
		\qquad
		\kappa_n :=8\cos^2\frac\pi n-2 \,.
		\end{equation}
		We have
		$$
		0<\sigma_n < \frac34
		\qquad\text{and, if}\ n\ge 6 \,, \text{then}\
		\kappa_n\ge4 \,.
		$$
	\end{lemma}
	
	\begin{proof}
        To prove the claimed bounds, we shall use the discrete Fourier transform or, alternatively, the spectral resolution of the graph Laplacian. 
        
        For $x:C_n\to\mathbb R$ we define the Fourier coefficients
$$
\widehat{x}_k := \avg{\overline{\chi^{(k)}} \, x} \,,
$$	
and note that Fourier inversion identity and Parseval relation
$$
x = \sum_{k=0}^{n-1} \widehat{x}_k \, \chi^{(k)} \,,
\qquad
\| x\|_2^2 = \sum_{k=0}^{n-1} |\widehat{x}_k|^2 \,.
$$
Moreover, we have
$$
D(x) = \avg{x(Lx)} = \sum_{k=0}^{n-1} \mu_k |\widehat{x}_k|^2 \,.
$$	

        If \(z\perp \one,\, V_1\), then $\widehat{z}_k=0$ for $k=0,1,n-1$ and therefore
        $$
        D(z) = \sum_{k=2}^{n-2} \mu_k |\widehat{z_k}|^2
        $$
        and
        $$
        Q(z) = \sum_{k=2}^{n-2} \left( \frac{\mu_k}{\lambda_n} - 2 \right) |\widehat{z}_k|^2
        $$

        The smallest coefficient
		\[
		\frac{\mu_k}{\lambda_n}-2,
		\qquad k\notin\{0,1,n-1\},
		\]
		is attained at \(k=2\) and \(k=n-2\). Hence
		\[
		Q(z)\ge
		\left(\frac{\mu_2}{\lambda_n}-2\right) \sum_{k=2}^{n-2} |\widehat{z}_k|^2 = \left(\frac{\mu_2}{\lambda_n}-2\right) \|z\|_2^2 \,.
		\]
		We have
		\[
		\frac{\mu_2}{\lambda_n}-2
		=
		\frac{2(1-\cos(4\pi/n))}{1-\cos(2\pi/n)}-2
		=
		8\cos^2\frac\pi n-2.
		\]
		For \(n\ge6\), this is at least \(4\). This proves the second bound in the lemma and the bound on $\kappa_n$.

        To prove the first bound in the lemma, we use the Fourier inversion identity and bound, using the Cauchy--Schwarz inequality,
		\[
		|z_j|^2 = \left| \sum_{k=2}^{n-2} \widehat{z}_k \chi^{(k)}_j \right|^2
		\le
		\left(
		\sum_{k=2}^{n-2}
		\frac1{\frac{\mu_k}{\lambda_n}-2}
		\right)Q(z) \,.
		\]
		It remains to compute the sum. Put \(\theta=\pi/n\). Since
		\[
		\lambda_n=1-\cos(2\theta)=2\sin^2\theta,
		\]
		and
		\[
		\cos(2\theta)-\cos(2k\theta)
		=
		2\sin((k-1)\theta)\sin((k+1)\theta),
		\]
		we get
		\[
		\frac1{\frac{\mu_k}{\lambda_n}-2}
		=
		\frac{\sin^2\theta}{
			2\sin((k-1)\theta)\sin((k+1)\theta)}.
		\]
		Using
		\[
		\cot((k-1)\theta)-\cot((k+1)\theta)
		=
		\frac{\sin(2\theta)}
		{\sin((k-1)\theta)\sin((k+1)\theta)},
		\]
		we obtain
		\[
		\frac1{\frac{\mu_k}{\lambda_n}-2}
		=
		\frac{\tan\theta}{4}
		\left[
		\cot((k-1)\theta)-\cot((k+1)\theta)
		\right].
		\]
		Summing from \(k=2\) to \(k=n-2\) gives
		\[
		\begin{aligned}
            \sum_{k=2}^{n-2} \frac1{\frac{\mu_k}{\lambda_n}-2}
            & =	\frac{\tan\theta}{4}
			\sum_{k=2}^{n-2}
			\left[
			\cot((k-1)\theta)-\cot((k+1)\theta)
			\right]                                      \\
			&=
			\frac{\tan\theta}{2}
			\left(\cot\theta+\cot(2\theta)\right)        \\
			&=
			\frac34-\frac14\tan^2\theta.
		\end{aligned}
		\]
		This proves the first bound in the lemma.
	\end{proof}
	
%%%%%%%%%%%%%%%%%%%%%%%%%%%%
%%%%%%%%%%%%%%%%%%%%%%%%%%%%	
	
	\subsection*{Proof of Theorem \ref{thm:cubic}}
	
	\begin{proof}
		Let $x:C_n\to\mathbb R$ with $x_j\ge 0$ for all $j\in C_n$ and $\ip{x^2}=1$. Set
		$$
		a:= \ip{x}
		$$
		and note that $a\ge 0$ as a consequence of $x_j\geq 0$ for all $j$. We decompose
		$$
		x = a + v + z
		$$
		where
		\[
		v\in V_1,
		\qquad
		z\perp \one,
		\qquad
		z\perp V_1 \,,
		\]
		By orthogonality, it follows that
		$$
		1 = \| x\|_2^2 = a^2 + \|v\|_2^2 + \|z\|_2^2
		$$
		and
		$$
		D(x) = 2\lambda_n \|v\|_2^2 + D(z) \,.
		$$
		It is convenient to introduce the quantity
		$$
		Q:=Q(z)= \frac{1}{\lambda_n} D(z) - 2 \|z\|_2^2 \,,
		$$
		which is nonnegative by Lemma \ref{lem:highfreq}. Writing
        $$
        D(x) = \lambda_n \left( Q + 2(\|v\|_2^2 + \|z\|_2^2) \big) = \lambda_n \big( Q + 2(1-a^2) \right),
        $$
		we see that the claimed inequality in Theorem \ref{thm:cubic} can be expressed as
		\begin{align}\label{eq:goalsob}
			Q & \ge \frac23 \ip{(x-1)^2(x+2)} - 2 (1-a^2) \,.
		\end{align}
		
		Let us compute the nonlinear term. We have
		\begin{align*}
			(x-1)^2(x+2)
			& = (a-1)^2(a+2) +(a-1)^2 (v+z) \\
			& \quad + 2(a-1)(a+2)(v+z) + 2(a-1)(v+z)^2 \\
			& \quad + (a+2)(v+z)^2 +(v+z)^3
		\end{align*}
		and consequently
		\begin{align*}
			\ip{(x-1)^2(x+2)} & = (a-1)^2(a+2) + 3a \ip{(v+z)^2} + \ip{(v+z)^3} \,.
		\end{align*}
        The orthogonality relations lead to
		$$
		\ip{(v+z)^2} = \ip{v^2} + \ip{z^2} = 1-a^2
		$$
		and we find
		$$
			\ip{(x-1)^2(x+2)} = (a-1)^2(a+2) + 3a(1-a^2) + \ip{(v+z)^3} \,.
		$$
		Inserting this into \eqref{eq:goalsob}, we see that the inequality in Theorem \ref{thm:cubic} becomes
		\begin{align}\label{eq:goalsob2}
			Q & \geq \frac23 \left( - (1-a)^2(1+2a) + \ip{(v+z)^3} \right).
		\end{align}
		
		To proceed, we need to bound the term $\ip{(v+z)^3}$ from above. We now treat three different cases, depending on whether $n=4$, $n=5$ or $n\geq 6$. It will be convenient to use the abbreviations
		$$
		r:= \|v\|_2 \,,
		\qquad
		t:= \|z\|_2 \,.
		$$

		\medskip
		\noindent
		\textbf{Case \(n=4\).}
        In this case the orthogonal complement of $\one$ and $V_1$ is one-dimensional and spanned by what is sometimes known as Nyquist mode. Thus, there is a $c\in\mathbb R$ such that
		\[
		z_j=c \,(-1)^j.
		\]
		Then \(t=|c|\). Write
		\[
		v_j=p\cos\frac{\pi j}{2}
		+
		q\sin\frac{\pi j}{2}.
		\]
		Then the values of \(v\) on the four points of $C_4$ are
		\[
		p,q,-p,-q,
		\]
		and hence
		\[
		r^2= \|v\|_2^2 = \frac{p^2+q^2}{2}.
		\]
		Also
		\[
		\ip{v^3}=0,
		\qquad
		\ip{vz^2}=0,
		\qquad
		\ip{z^3}=0,
		\]
		and
		\[
		|\ip{v^2z}|
		=
		\frac{|c|}{2} \, |p^2-q^2|
		\le
		|c| \, \frac{p^2+q^2}{2}
		=
		tr^2.
		\]
		Thus
		\[
		\ip{(v+z)^3} = \ip{v^3} + 3 \ip{v^2z} + 3 \ip{vz^2} + \ip{z^3} \le 3r^2t.
		\]
        Using the elementary inequality \eqref{eq:scalar3}, we deduce
        $$
        \ip{(v+z)^3} \le (1-a)^2 (1+2a) + 3 t^2
        $$
        and therefore
		$$
		\frac23 \left( -(1-a)^2(1+ 2a) + \ip{(v+z)^3} \right) \leq 2 t^2 \,.
		$$
        For $n=4$ we have (see Lemma \ref{lem:highfreq} and its proof)
        $$
        Q = 2t^2 \,,
        $$
        which proves the desired inequality \eqref{eq:goalsob2}. This proves Theorem \ref{thm:cubic} for $n=4$.
		
		\medskip
		\noindent
		\textbf{Case \(n=5\).}
		Define $\chi:C_n\to\mathbb C$ by
		\[
		\chi_j :=e^{2\pi i j/5} \,.
		\]
        Then $\chi=\chi^{(1)}$ and $\chi^{-1}=\chi^{(4)}$ span $V_1$, and $\chi^2$ and $\chi^{-2}$ span the orthogonal complement of constants and $V_1$. Thus, there are $A,B\in\mathbb C$ such that
		\[
		v=A\chi+\overline A\chi^{-1},
		\qquad
		z=B\chi^2+\overline B\chi^{-2}.
		\]
		Then
		\[
		r^2=2|A|^2,
		\qquad
		t^2=2|B|^2.
		\]
		A direct multiplication modulo \(5\) gives
		\[
		\ip{(v+z)^3}
		=
		6\operatorname{Re}\left(A^2\overline B+AB^2\right),
		\]
		and consequently
		\[
		\ip{(v+z)^3}
		\le
		\frac{3}{\sqrt2}(r^2t+rt^2) \,.
		\]
		Using the elementary inequality \eqref{eq:scalar2}, we deduce
		$$
		\ip{(v+z)^3} \leq (1-a)^2 (1+2a) + \frac52 t^2 \,,
		$$
		and therefore
		$$
		\frac23 \left( -(1-a)^2(1+ 2a) + \ip{(v+z)^3} \right) \leq \frac53 \, t^2 \,.
		$$
		For $n=5$ we have (see Lemma \ref{lem:highfreq} and its proof)
		$$
		Q = (1+\sqrt 5) \, t^2 \,,
		$$
		so the desired inequality \eqref{eq:goalsob2} follows from
		$$
		1+\sqrt 5 \geq \frac53 \,,
		$$
		This proves Theorem \ref{thm:cubic} for $n=5$.
		
		\medskip
		\noindent
		\textbf{Case \(n\ge6\).}		
		For \(v\in V_1\), we write
		\[
		v=A\chi+\overline A\chi^{-1}
		\]
        with some $A\in\mathbb C$. Here $\chi := \chi^{(1)}$ from the proof of Lemma \ref{lem:highfreq}. A straightforward computation gives
        \[
		\ip{v^3}=0 \,,
		\qquad
		\|v\|_\infty \le \sqrt2\,r \,,
		\qquad
		\|v^2-\ip{v^2}\|_2=\frac{r^2}{\sqrt2} \,.
		\]
		Therefore,
		\[
		|\ip{v^2z}| = |\ip{(v^2 - \ip{v^2}) z}|
		\le \|v^2-\ip{v^2}\|_2 \|z\|_2 
        \le \frac{r^2}{\sqrt2} \,t
        \]
        and
        \[
		|\ip{vz^2}|
		\le \|v\|_\infty \|z\|_2^2 \le \sqrt2\,rt^2.
		\]
		Moreover, by Lemma~\ref{lem:highfreq},
		\[
		|\ip{z^3}|
		\le
		\|z\|_\infty \|z\|_2^2
		\le
		t^2\sqrt{\sigma_n} \sqrt Q \,.
		\]
		Thus,
		\begin{align*}
			\ip{(v+z)^3} = 3\ip{v^2 z} + 3\ip{vz^2} + \ip{z^3} \leq \frac 3{\sqrt2} r^2 t + 3\sqrt 2 r t^2 + \sqrt{\sigma_n} \sqrt Q \, t^2 \,.
		\end{align*}
		Using the elementary inequality \eqref{eq:scalar1}, we deduce
		\begin{align*}
			\ip{(v+z)^3} \leq (1-a)^2 (1+2a) + 4 t^2 + \sqrt{\sigma_n} \sqrt Q \, t^2
		\end{align*}
		and therefore
		$$
		\frac23 \left( -(1-a)^2(2a+1) + \ip{(v+z)^3} \right) \leq \frac83 t^2 + \frac23 \sqrt{\sigma_n} \sqrt Q t^2 \,.
		$$
		Thus, the desired inequality \eqref{eq:goalsob2} follows if we can show that
		\begin{equation}
			\label{eq:goalsob3}
					Q \geq \frac83 t^2 + \frac23 \sqrt{\sigma_n} \sqrt Q t^2 \,.
		\end{equation}
		Let us prove this. By Lemma~\ref{lem:highfreq},
		\[
		Q\ge\kappa_nt^2,
		\qquad
		\kappa_n\ge4,
		\qquad
		\sigma_n<\frac34.
		\]
		Thus,
		$$
		Q = \frac23 Q + \frac13 Q \geq \frac23\kappa_n t^2 + \frac13 \sqrt Q \sqrt\kappa_n t \,.
		$$
		Since $\kappa_n\ge 4$, we have $\frac23\kappa_n\ge \frac 83$ and, using also $t\leq 1$, $\frac13\sqrt{\kappa_n} \ge \frac23 > \frac23 \sqrt{\sigma_n} t$. This proves \eqref{eq:goalsob3}.
		
		This proves Theorem \ref{thm:cubic} for $n\ge 6$ and completes the proof.
	\end{proof}

%%%%%%%%%%%%%%%%%%%%%%%%%%%%%%%%%
%%%%%%%%%%%%%%%%%%%%%%%%%%%%%%%%%

    \section*{Acknowledgments}

R.~L.~F.~acknowledges partial support from US NSF grant DMS-1954995 and the DFG grants EXC-2111-390814868 and TRR 352-Project-ID 470903074. P.~I.~acknowledges partial support from the US NSF CAREER grant DMS-2152401, US NSF grant DMS-2554183,  a Simons Fellowship, and a Humboldt Research Fellowship for Experienced Researchers. The authors acknowledge the use of AI tools during the exploratory stage of this project. All mathematical arguments and proofs in the final manuscript were checked and written by the authors.

\end{document}